\documentclass[pdflatex]{amsart}

\usepackage{amsmath}
\usepackage{amssymb}
\usepackage{latexsym}
\usepackage{amscd,amsthm}
\usepackage[small,it]{caption }
\usepackage{booktabs}
\usepackage{tabularx}
\usepackage{subfigure}
\usepackage{paralist}
\usepackage{bm}
\usepackage{tikz}

\newcommand{\group}[1]{\langle\,#1\,\rangle}
\newcommand{\divides}{\,|\,}

\newcommand{\card}[1]{{\mid\! #1 \!\mid}}
\newcommand{\pro}[2]{\langle\, #1, #2\, \rangle}

\newtheorem*{theorem*}{Theorem}
\newtheorem*{corollary*}{Corollary}
\newtheorem{theorem}{Theorem}[section]
\newtheorem{example}[theorem]{Example}
\newtheorem{corollary}[theorem]{Corollary}
\newtheorem{prop}[theorem]{Proposition}
\newtheorem{lemma}[theorem]{Lemma}
\newtheorem{conjecture}[theorem]{Conjecture}

\theoremstyle{definition}
\newtheorem{definition}[theorem]{Definition}
\newtheorem{remark}[theorem]{Remark}

\DeclareMathOperator{\prmat}{\mathsf M}

\DeclareMathOperator{\End}{End}
\DeclareMathOperator{\supp}{supp}
\DeclareMathOperator{\aff}{aff}
\DeclareMathOperator{\lin}{lin}
\DeclareMathOperator{\degree}{deg}
\DeclareMathOperator{\lcm}{lcm}
\DeclareMathOperator{\Irr}{Irr}
\DeclareMathOperator{\conv}{conv}
\DeclareMathOperator{\GL}{GL}
\let\mod\relax
\DeclareMathOperator{\mod}{mod}
\newcommand\Mat{\textup{Mat}}

\newcommand{\vlx}{\lambda^{(x)}}
\newcommand{\vly}{\lambda^{(y)}}
\newcommand{\vlz}{\lambda^{(z)}}
\newcommand\C{{\mathbb C}}
\newcommand\N{{\mathbb N}}

\newcommand\R{{\mathbb R}}
\newcommand\Z{{\mathbb Z}}

\begin{document}
\title{Permutation Polytopes of Cyclic Groups}

\author[Baumeister]{Barbara Baumeister}
\address{Barbara Baumeister, Universit\"at Bielefeld, Germany}
\email{b.baumeister@math.uni-bielefeld.de}
\author[Haase]{Christian Haase}
\address{Christian Haase, Goethe-Universit\"at Frankfurt, Germany}
\email{haase@mathematik.uni-frankfurt.de}
\author[Nill]{Benjamin Nill}
\address{Benjamin Nill, Case Western Reserve University, Cleveland, OH, USA}
\email{benjamin.nill@case.edu}
\author[Paffenholz]{Andreas Paffenholz}
\address{Andreas Paffenholz, Technische Universit\"at Darmstadt, Germany}
\email{paffenholz@mathematik.tu-darmstadt.de}

\begin{abstract}
We investigate the combinatorics and geometry of permutation polytopes 
associated to cyclic permutation groups, i.e., the convex hulls of
cyclic groups of permutation matrices. We give formulas for their
dimension and vertex degree.
In the situation that the generator of the group consists of at most
two orbits, we can give a complete combinatorial description of the
associated permutation polytope. 
In the case of three orbits the facet structure is already quite
complex. For a large class of examples we show that there exist
exponentially many facets.
\end{abstract}

\maketitle

\section*{Introduction}

A {\em Permutation polytope} is the convex hull of a group of permutation 
matrices. We refer to the preceding article \cite{BHNP}
for some historical and motivational remarks. The 
most famous permutation polytope is the {\em Birkhoff polytope}, whose
vertex set is the entire set of $n \times n$-permutation matrices.
In \cite{BHNP} we proposed the systematic study of permutation
polytopes in their own right. We introduced suitable notion of
equivalences, 
studied the vertex-edge graph, products and free sums, 
and classified all permutation polytopes up to dimension four. 

In this article, we investigate permutation polytopes associated to
cyclic permutation groups.
In order to learn more about general permutation polytopes it seems to
be crucial to enhance our understanding of the convex hulls of
subgroups generated by only one element.
This boils down to the study of the elementary number theory of the
cycle structure of the generator permutation.
Already a relatively small input can generate fairly complicated
polytopes: take the group generated by a permutation which is the
product of three disjoint cycles of lengths $10$, $18$, $45$.
This leads to a $57$-dimensional polytope with $90$ vertices and
$15373$ facets whose vertex-edge graph is complete. 
This example is about as complex as we can handle computationally.
Still, using the structure of a permutation polytope it is possible to
determine important invariants of the polytope like the dimension and
the vertex degree in terms of the cycle lengths (see Section 2).
For groups generated by a permutation which is a product of at most
two cycles we can characterize the polytopes completely 
(Proposition~\ref{two}). However, the previous example indicates that
in the situation of three cycles the complexity of the facet structure
of these polytopes becomes enormous.
We show in Theorem~\ref{main-theo} that the number of facets in such a
specific situation grows indeed exponentially in the dimension. 

In many respects, our experience has turned out to be similar to the
challenges faced by Hood and Perkinson \cite{JH04} when investigating
the facets of the permutation polytope associated to the group of even
permutations. They also constructed exponentially many facets with
respect to the dimension of the polytope.
However, in their situation the number of vertices grows
exponentially as well, while in our case the number of vertices
remains polynomially bounded.

The features of many of these objects such as a large number of facets
and a complete vertex-edge graph are reminiscent of the properties of
cyclic polytopes \cite[pp.10-16]{Ziegler}. While the latter ones are
simplicial, in many of the cases considered here, each facet contains
far more than half of the total number of vertices of the polytope.
Permutation polytopes of cyclic permutation groups might be
considered as highly symmetric analogues of cyclic polytopes.
It was recently shown by Rehn~\cite{Rehn10} that if the order of a
cyclic permutation group is $n=k_1 \cdots k_r$, where $k_1, \ldots,
k_r$ are coprime prime powers, then the associated permutation
polytope has at least $k_1! \cdots k_r!$ many affine automorphisms. 
On the other hand, Kaibel and Wa\ss mer~\cite{KW11} show that the
order of the combinatorial automorphism group of a cyclic polytope is
at most twice its number of vertices.

Cyclic permutation polytopes -- and more generally abelian permutation
polytopes -- are instances of so-called marginal polytopes. Their
inequality description is important in statistics and optimization.
This will be explored in an upcoming paper~\cite{Marginalspaper}
(cf.~Remark~\ref{comment}). 

\subsection*{Note} One should not confuse `permutation polytopes' with
`orbitopes', the convex hull of an orbit of a compact group acting
linearly on a vector space.
Recently, Sanyal, Sottile and Sturmfels \cite{Orbitopes} gave a
systematic approach to orbitopes. They also studied the permutation
polytopes associated to the groups $O(n)$ and $SO(n)$.
In this setting permutation polytopes are called {\em tautological
  orbitopes}.
Since for each orbitope there is a permutation polytope mapping
linearly onto it, permutation polytopes serve as {\em initial objects}
in this context.

\subsection*{Organization of the paper.} In Section~1 we introduce notation and 
basic properties. In Section~2 we give formulas for the dimension and the 
vertex degree, and we describe a criterion when a vertex forms an edge
with the unit of the group. 
In Section~3 we study closely the situation when the group generator
is decomposed in at most three cycles. While we can completely
describe the case of the one or two cycles, the first difficult
situation occurs for three cycles, where we construct a large family
of facets for one infinite class of examples.

\subsection*{Acknowledgments} Many of these results are based on
extensive calculations using the software packages GAP~\cite{GAP4} and
polymake~\cite{polymake99}. The last three authors were supported by
Emmy Noether fellowship HA 4383/1 of the German Research Foundation
(DFG). The third author is supported by the US National Science Foundation 
(DMS 1102424). The last author is supported by the DFG Priority Program 1489.

\section{Notation and Basic Properties}

\subsection{Notation}

For a positive integer $n \in \N$ we denote 
\begin{align*} 
[n] &:= \{1, \ldots,n\}. 
\intertext{Since it will be more suitable later on, we also define}
[[n]] &:= \{0, \ldots, n-1\}.
\end{align*}
For a finite set $I \subset \N$ we denote by $\gcd(I)$ and $\lcm(I)$
the greatest common divisor and the least common multiple of all
elements in $I$, respectively.  By convention $\gcd(\emptyset) := 0$
and $\lcm(\emptyset) := 1$. For integers $k,l\in\Z$ we write $k
\divides l$ if $k$ divides $l$.

The convex and the affine hull of a set $S$ in a real vector space
will be denoted by $\conv(S)$ and by $\aff(S)$, respectively.

\subsection{Representation polytopes}
Let $V$ be a real $n$-dimensional vector space. Then $\GL(V)$ denotes
the set of automorphisms.
By choosing a basis we can identify $\GL(V)$ with the set $\GL_n(\R)$
of invertible $n\times n$-matrices.
In the same way, we identify $\End(V)$ with the vector space
$\Mat_n(\R)$ of $n\times n$-matrices.

Let $G$ be a group. A homomorphism $\rho \colon G \to \GL(V)$ is
called a {\em real representation}.
In this case
\[(\rho) := \conv(\rho(g) \;:\; g \in G) \ \subseteq \Mat_n(\R) \cong \R^{n^2}\]
is called the associated {\em representation polytope}.

\subsection{Permutation polytopes}
The symmetric group $S_n$ acts on the set $[n]$. By identifying $[n]$ with 
the basis vectors $\{e_1, \ldots, e_n\}$ of $\R^n$, we get a representation $S_n \to \GL(\R^n)$. 
This map identifies the symmetric group $S_n$ 
with the set of $n \times n$ {\em permutation matrices}, i.e., the set of matrices with entries $0$ or $1$ 
such that in any column and any row there is precisely one $1$. 
For a subset $G\subseteq S_n$ we let $\prmat(G)$ be the corresponding set of permutation matrices. 
For $S \subseteq G$ we let $\group S$ be the smallest subgroup of $G$ containing $S$.

An injective homomorphism $G \to S_n$ is called {\em permutation representation}. 
Subgroups $G\le S_n$ are called {\em permutation groups}. 
In this case, the representation polytope
\[P(G) := \conv(\prmat(G))\]
is called the {\em permutation polytope} associated to $G$.

The special case $G=S_n$ yields the well-known $n$th {\em Birkhoff
  polytope} $B_n:=P(\prmat(S_n))$ (see e.g.\ \cite{BS96}). It has
dimension $(n-1)^2$.

\subsection{Equivalences}

When working with permutation polytopes, one would like to identify
permutation groups that clearly define affinely equivalent permutation
polytopes.  Therefor, we introduced in~\cite{BHNP} the notion of
stable equivalence. Here, $\R[G]$ denotes the group algebra of $G$
with real coefficients.

\begin{definition}
  For a representation $\rho \colon G \to \GL(V)$ define the affine
  kernel $\ker^\circ \rho$ as
  $$\ker^\circ \rho := \left\{ \ \sum_{g \in G} \lambda_g g \in \R[G] \ : \
    \sum_{g \in G} \lambda_g \rho(g) = 0 \text{ and } \sum_{g \in G} \lambda_g = 0
    \ \right\}$$

Say that a real representation $\rho' \colon G \to \GL(V')$ is an {\em affine quotient} of $\rho$ if $\ker^\circ \rho \subseteq \ker^\circ \rho'$. Then real representations $\rho_1$ and $\rho_2$ of $G$ are {\em stably equivalent}, if there are affine quotients $\rho_1'$ of $\rho_1$ and $\rho_2'$ of $\rho_2$ such that $\rho_1 \oplus \rho_1' \cong \rho_2 \oplus \rho_2'$ as $G$-representations.
\end{definition}

\begin{example}
  The following representations of the group $\Z_4$ are stably   equivalent:
  \begin{align*}
    \langle (1 2 3 4)\rangle &\leq S_4\,,& \langle (1 2 3 4)     (5)\rangle &\leq S_5\,,\\
    \langle (1 2 3 4) (5 6)\rangle &\leq S_6\,,& \langle (1 2 3 4) (5     6) (7 8)\rangle &\leq S_8\,,\\
    \langle (1 2 3 4) (5 6 7 8) \rangle &\leq S_8\,.
  \end{align*} 
\end{example}

For the following, let us 
denote by $\Irr(G)$ the set of pairwise non-isomorphic irreducible $\C$-representations, i.e., homomorphisms
$G \to \GL(W)$ where $W$ is a $\C$-vector space which does not contain a proper $G$-invariant subspace. 
For instance, there is the {\em trivial representation}, $1_G$: $G \to \GL(\C)$, $g \mapsto 1$. 
As a $G$-representation over $\C$ any real representation $\rho: G \to \GL(V)$ splits into 
irreducible representations. We denote these {\em irreducible factors} of $\rho$ by $\Irr(\rho) \subseteq \Irr(G)$.

In \cite{BHNP} we proved an explicit criterion for the polytopes of
two representations to be stably equivalent.

\begin{theorem}[{Baumeister et~al.~\cite[2.3]{BHNP}}]\label{thm:stablyEquivalence}
  Suppose $\rho$ and $\bar{\rho}$ are stably equivalent real representations of a finite group $G$. 
Then $P(\rho)$ and $P(\bar{\rho})$ are affinely equivalent.

Two real representations are stably equivalent if and only if they contain the same non-trivial irreducible factors.
\end{theorem}

\begin{definition}
Two real representations $\rho_i \colon G_i \to \GL(V_i)$ (for $i=1,2$) of finite groups are {\em effectively equivalent}, if there exists an isomorphism $\phi \colon G_1 \to G_2$ such that $\rho_1$ and $\rho_2 \circ \phi$ are stably equivalent $G_1$-representations. 

Moreover, we say $G_1 \leq S_{n_1}$ and $G_2 \leq S_{n_2}$ are {\em effectively equivalent} permutation groups, if $G_1 \hookrightarrow S_{n_1}$ and $G_2 \hookrightarrow S_{n_2}$ are effectively equivalent permutation representations.
\end{definition}

By Theorem \ref{thm:stablyEquivalence} two permutation groups are
effectively equivalent if they are isomorphic as abstract groups such
that via this isomorphism the permutation representations contain the
same non-trivial irreducible factors. In particular, the associated
permutation polytopes are affinely equivalent.

The vector space $\Mat_n(\R)$ in which permutation polytopes live
comes with a natural lattice $\Mat_n(\Z)$ of integral matrices.
For polytopes with vertices in a lattice -- such as permutation
polytopes -- we can ask whether an affine equivalence preserves the
lattice. In that case we call the polytopes lattice equivalent.
Lattice equivalence of permutation polytopes is a subtle issue --
cf.~\cite[Example 2.9]{BHNP}.

\subsection{Dimension formula}

Let us recall that the degree of a representation is the dimension of
the vector space the group is acting on.  Guralnick and
Perkinson~\cite{GP05} determined the dimension of the polytope
associated to a representation of a group.

\begin{theorem}[{Guralnick and Perkinson~\cite[Thm.\ 3.2]{GP05}}]\label{thm:dimension}
Let $G\le S_n$ be a permutation group and $\rho$ a representation of $G$. Then
\begin{align*} 
\dim P(\rho) = \sum_{1_G \neq \sigma \in \Irr(\rho)}   (\degree \sigma)^2\,.
\end{align*}
\end{theorem}

\subsection{Indecomposable elements}

Every vertex of $P(G)$ corresponds bijectively to a group element of $G$. 
For the edges of $P(G)$ there is an explicit description. It was used by
Guralnick and Perkinson~\cite{GP05} to determine the diameter of a
permutation polytope.

\begin{definition}
Let $e \not= g \in G$.
\begin{itemize}
\item We denote by $F_g$ the smallest face of $P(G)$ containing the identity $e$ and $g$.
\item We denote by $g = z_1 \circ \cdots \circ z_r$ the unique \emph{disjoint cycle decomposition} of $g$ in $S_n$, i.e., $z_1, \ldots, z_r$ are cycles with pairwise 
disjoint support, and $g = z_1 \cdots z_r$.
\item Let $g = z_1 \circ \cdots \circ z_r$. For $h \in S_n$ we say $h$ is a \emph{subelement} of $g$ (we write $h \preceq g$), 
if there is a set $I \subseteq [r]$ such that $h = \prod_{i \in I} z_i$.
\item $g$ is called \emph{indecomposable} in $G$, if $e$ and $g$ are the only subelements of $g$ in $G$.
\end{itemize}
\end{definition}
With this definition one can characterize the faces $F_g$.
\begin{theorem}[{Guralnick and Perkinson~\cite[Thm.\ 3.5]{GP05}}]\label{smallfacetheo}
Let $g \in G$. The vertices of $F_g$ are precisely the subelements of $g$ in $G$. 
In particular, $e$ and $g$ form an edge of $P(G)$ if and only if $g$ is indecomposable in $G$. 
\end{theorem}
The {\em degree} of a vertex of a polytope is the number of edges it is contained in. 
Since $G$ acts transitively on the vertices of $P(G)$ each vertex has the same degree.
\begin{corollary}\label{cor:degreecoro}
The number of indecomposable elements (different from $e$) in $G$ equals the degree of any vertex of $P(G)$.
\end{corollary}

\subsection{Products}

For the purpose of reference, let us cite the following result
concerning products of permutation polytopes. Here, the support
$\supp(H)$ of a permutation group $H \le S_n$ is the
complement in $[n]$ of the set of fixed points of $H$.

\begin{theorem}[{Baumeister et~al.~\cite[3.5]{BHNP}}]\label{product}
$P(G)$ is a combinatorial product of two polytopes $P_1$ and $P_2$ 
if and only if there are subgroups $H_1$ and $H_2$ in $G$ such that 
\begin{enumerate}
\item $P(H_i)$ is combinatorially equivalent to $P_i$ for $i = 1,2$.
\item $\supp(H_1) \cap \supp(H_2) = \emptyset$
\item $G = H_1 \times H_2$.
\end{enumerate}
\end{theorem}

\section{Dimension and vertex degree}

\subsection{Our setting}

In this section, we give formulas for the dimension and the vertex
degree of cyclic permutation polytopes in terms of the cycle type of
the generator permutation.
Let $G = \group g$, where $g$ has a disjoint cycle decomposition 
into $t$ cycles of lengths $\ell_1, \ldots, \ell_t$. 
In this case, we set
\[d := \card{G} = o(g) = \lcm(\ell_1, \ldots, \ell_t),\]
so $G = \{e,g,\ldots, g^{d-1}\}$. 

\subsection{Dimension formula}

In our setting, we can explicitly determine the dimension of $P(G)$.
\begin{prop}\label{dimsimplex}
Let $G=\group{g}$ be a cyclic permutation group of order $d$, where
$g$ has $t$ disjoint cycles of lengths $\ell_1, \ldots, \ell_t$.
We have two ways to compute the dimension of $P(G)$:
\begin{itemize}
 \item[(1)] $\dim(P(G))$ equals the number of ${\ell_i}^{th}$-roots of
   unity, $i \in [t]$, which are different from $1$, i.e.
\[\dim(P(G))~ = \left|\left\{x \in \C : x \neq 1~\mbox{and}~x^{\ell_i}
    = 1~\mbox{for some}~i \in [t] \right\}\right|.\]
\item[(2)] 
$\dim(P(G))$ equals the number of elements in $[d-1]$ which are
divisible by $\frac{d}{\ell_i}$ for some $i \in [t]$. 
\end{itemize}
In particular, 
\[\max(\ell_i-1 \,:\, i \in [t]) \leq \dim(P(G)) \leq d-1.\]
\end{prop}
\begin{proof}
We consider the permutation representation of $G$ over $\C$. As $G$ is cyclic, it splits over $\C$ in $1$-dimensional
representations. The eigenvalues of $G$ are then the $\ell_i$-th roots of unity. In order to determine the number
of non-isomorphic non-trivial irreducible representations it suffices to count the different non-trivial $\ell_i$-th roots of unity,
which shows (1). Part (2) follows from Theorem~\ref{thm:dimension} and by observing that the subgroup $\{x \in \C \;:\: x^d = 1\}$ 
is generated by a primitive $d$'th root of unity.
\end{proof}
For instance, for $\ell_1=2,\ell_2=4,\ell_3=8$ we get $\dim(P(G)) =7$, cf. Corollary~\ref{simplexkrit}. 
Note that the first criterion is more conceptual, while the second one is easier to implement. 
Using the inclusion-exclusion-formula we obtain a closed formula for the dimension:
\begin{corollary}\label{formula}
Under the assumptions of the proposition we have:
\[\dim(P(G)) = -1 + \sum_{\emptyset \not= I \subseteq [t]} (-1)^{\card{I}+1} \gcd(\ell_i \,:\, i \in I).\]
\end{corollary}
\begin{proof}
Let $\emptyset \not= I \subseteq [t]$. Applying the inclusion-exclusion formula to Proposition~\ref{dimsimplex}(2) 
we only have to observe that the set 
\begin{align*}
  \{k \in [d-1] \,:\, \lcm(\frac{d}{\ell_i} \,:\, i \in I) \,|\, k\}
\end{align*}
has cardinality $\gcd(\ell_i \,:\, i \in I) - 1$. Note that 
\begin{align*}
  \sum_{\emptyset \not= I \subseteq [t]} (-1)^{\card{I}+1} (-1) &=
  \sum_{a=1}^t \binom{t}{a} (-1)^a = -1.\qedhere
\end{align*}
\end{proof}
Since a polytope is a simplex if and only if $\dim(P(G))+1$ equals the number of vertices (here, $|G|=d$) we get from 
Proposition~\ref{dimsimplex} the following criterion (cf. \cite{GP05}). 
Let us define the {\em unimodular $m$-simplex} $\Delta_m$ as the convex hull of the standard basis vectors 
in $\R^{m+1}$.
\begin{corollary}\label{simplexkrit}
Let $G$ be cyclic with $\card{G} = d$. Then $P(G)$ is a simplex if and
only if $G$ has a cycle of order $d$.
In this case, there is an isomorphism $\Z^{n^2} \cap \aff(P(G)) \to
\Z^d \cap \aff(\Delta_{d-1})$ mapping the vertices of $P(G)$ onto the
vertices of $\Delta_{d-1}$. In other words, $P(G)$ is a unimodular
simplex up to lattice isomorphisms.
In particular this holds, if $\card{G}$ is a prime power.
\end{corollary}
\begin{proof}
We observe that $\dim(P(G)) = d-1$ if and only if $1$ is in the set
given in the Proposition~\ref{dimsimplex}(2), or, equivalently, if and only if 
there is an $i \in [t]$ such that $d/\ell_i = 1$. 

For the additional statement, we may assume that the first cycle in
the cycle decomposition of $g$ is of the form $(1 \cdots d)$.
Let $e_1, \ldots, e_d$ be the canonical basis of $\Z^d$. Then by
projecting onto the first $d$ coordinates of the first row
the permutation matrix $\prmat(g^{i-1})$ gets mapped to $e_i$ (for
$i \in [d]$). Since the inverse map given by mapping $e_i$ to
$\prmat(g^{i-1})$ for $i = 1, \ldots, d$ is affine and integral, this
yields a lattice isomorphism $P(G) \cap \Z^{n^2} \to \aff(e_1, \ldots,
e_d) \cap \Z^d$.
In particular, $P(G)$ is isomorphic to the convex hull of the $d$
canonical basis vectors of $\R^d$, a unimodular
simplex.
\end{proof}

In particular, since Ehrhart polynomials and volume of unimodular simplices are well-known, this gives an immediate proof 
of Theorem 1.2(1) and Lemma 3.1 in \cite{BDO11}. 

Notice that Corollary~\ref{simplexkrit} also follows directly from Corollary~2.8 of [BHNP]: 
According to that corollary $P(G)$ is a simplex if and only if the
permutation representation of $G$ contains every irreducible complex
representation of $G$. Hence, if $P(G)$ is a simplex, then  all the
$d$'th roots of unity are eigenvalues of this representation. This
implies that $G$ has a cycle of order $d$.

Here is another special situation, which is a generalization of \ref{simplexkrit}.

\begin{prop}\label{disjoint}
Let $G = \group g \leq S_n$ be a cyclic permutation group where the orders $\ell_1, \ldots, \ell_t$ of the disjoint cycles of $g$ are pairwise coprime. 
Then $P(G)$ is a product of unimodular simplices of dimensions $\ell_1 - 1, \ldots, \ell_t - 1$.
\end{prop}
\begin{proof}
The Chinese remainder theorem implies that $G$ is isomorphic to the product of cyclic permutation groups (with disjoint support) of orders $\ell_1, \ldots, \ell_t$. From the previous corollary and Theorem~\ref{product} the statement follows.
\end{proof}

\subsection{The number of indecomposable elements}

We will give a criterion to determine whether $z \in G  = \langle g \rangle = \{e, g, \ldots, g^{d-1}\}$ is decomposable or not. 
Let $g = z_1 \circ \cdots \circ z_t$ be the cycle decomposition into $t$ cycles of 
lengths $\ell_1, \ldots, \ell_t$. For $I \subset [t]$ we set $$I^c := [t]\backslash I, \text{ and } d_I:= \lcm(\ell_i~|~i \in I).$$
\begin{prop}\label{2criterion decomp}\label{2crit}
Let $ g^k$ ($k \in \{0, \ldots,  d-1\}$) be an element in the cyclic group $G = \group g \leq S_n$. 
Then $g^k$ is decomposable if and only if there is a proper non-empty subset $I$ of $[t]$ such that 
\begin{enumerate}
\item $\gcd(d_I, d_{I^c})$ divides $k$
\item neither $d_I$ nor $d_{I^c}$ divides $k$
\end{enumerate}
\end{prop}
\begin{proof}
Suppose that $z = g^k$ is not indecomposable. Then there exist $0 < r,s <d$ such that $z = g^r g^s$ 
where $g^r$ and $g^s$ have disjoint support. Therefore, for all $i \in [t]$ we get that in the group $\group {z_i}$ the element 
$z_i^k$ is decomposable in $z_i^r$ and $z_i^s$. Since in $\group {z_i}$ all elements are indecomposable due to Corollary~\ref{simplexkrit} 
and Theorem~\ref{smallfacetheo}, 
we have either $z_i^r = 1$ or $z_i^s = 1$. Therefore, we can find a proper non-empty subset $I$ of $[t]$ such that 
$z_i^r = 1$ for all $i \in I$, and $z_i^s = 1$ for all $i \in I^c$. This implies $d_I | r$ and $d_{I^c} | s$. 
Since $k \equiv r+s \ \ (\mod\, d)$ and $\lcm(d_I, d_{I^c}) = d$, we easily see that (1) and (2) hold.

Now, let us assume that there is a proper non-empty subset $I$ of
$[t]$ such that (1) and (2) hold.
Let $a^\prime$ and $b^\prime$ be integers such that $\gcd(d_I,
d_{I^c}) = a^\prime d_I + b^\prime d_{I^c}$.  Due to (1) we have $k =
ad_I + b d_{I^c}$ with $a:= a^\prime k/\gcd(d_I, d_{I^c}) \in \Z$ and
$b:= b^\prime k/\gcd(d_I, d_{I^c}) \in \Z$.  Let $a = q_1 (d/d_I) +
r_1$ with $q_1 \in \Z$ and $r_1 \in [[d/d_I]]$ and $b
= q_2 (d/d_{I^c}) + r_2$ with $q_2 \in \Z$ and $r_2 \in
[[d/d_{I^c}]]$. Then $k = (q_1+q_2) d + r_1 d_I + r_2 d_{I^c}$. We
set $0 \leq s:= r_1 d_I < d$ and $0 \leq t:= r_2 d_{I^c} < t$.  Hence,
$g^k = g^s g^t$.  
If $s = 0$, then $r_1 = 0$ and therefore $d_{I^c}$ divides $k$ in
contradiction to (2). In the same way we obtain $t \neq 0$.
This shows that $g^k$ is decomposable.
\end{proof}

As the following result shows, it is possible to determine whether a
given element $z \in G$ is indecomposable without having to know the
cycle decomposition of a generator $g$ of the cyclic group $G$.
\begin{corollary}\label{criterion decomp}
Let $z$ be an element in the cyclic group $ G \leq S_n$ having cycle
decomposition into $r$ non-trivial cycles of lengths $\ell_1, \ldots,
\ell_r$.

Then $z$ is decomposable if and only if there is a proper non-empty
subset $K$ of $[r]$ such that $\gcd(\ell_i,\ell_j) = 1$ for all $i \in
K$ and $j \in K^c$.
\end{corollary}
\begin{proof}
Let us observe that for a cycle $w$ of length $\ell$, the multiple
$w^k$ has a cycle decomposition into disjoint cycles of the same
length $\ell/\gcd(\ell,k)$. Therefore, for $z = g^k$ with $0 \le k <
d$ the second condition translates into the following statement: there
is a proper non-empty subset $I$ of $[t]$ such that
$\ell_i/\gcd(\ell_i,k) \not=1$ for some $i \in I$,
$\ell_j/\gcd(\ell_j,k) \not=1$ for some $j \in I^c$, and  
$\gcd(\ell_i/\gcd(\ell_i,k),\ell_j/\gcd(\ell_j,k)) = 1$ for all $i \in
I$ and $j \in I^c$.
Now, the proof follows by applying Proposition~\ref{2criterion decomp}.
\end{proof}

Given the lengths of cycles in the generating element  we can immediately determine the number
 of indecomposable elements of the group by using an obvious sieve method. By Corollary~\ref{cor:degreecoro} this allows to deduce an explicit formula for the constant 
vertex degree of the associated permutation polytope. We leave the proof to the reader.
\begin{corollary}
Let us set up the following notation:
\begin{itemize}
\item For $s = (2^t-2)/2$, let $I_1, \ldots, I_s$ denote the pairwise
  different partitions of $[t]$, i.e., $\emptyset \not= I_m \subsetneq
  [t]$ and $I_m \uplus I^c_m = [t]$ for $m \in [s]$;
\item for $M \subseteq [s]$, let $y_M$ be the least common multiple of $\gcd(d_{I_m},d_{I^c_m})$ for all $m \in M$;
\item for $T \subseteq N \subseteq [s]$, let $z_{N,T}$ be the least common multiple of 
$d_{I^c_n}$ (for all $n \in N \backslash T$) and of $d_{I_n}$ (for all $n \in T$).
\end{itemize}
Then the degree of a vertex in a permutation polytope associated to the cyclic group $G = \group g$ is given by
\[\sum_{M \subseteq [s]} \sum_{N \subseteq M} 
\sum_{T \subseteq N} (-1)^{\card{M}+\card{N}} \left(\frac{d}{\lcm(y_M,z_{N,T})} - 1 \right).\]
\end{corollary}

As another application of Proposition~\ref{2criterion decomp} we can characterize permutation polytopes 
with complete vertex-edge graph. 
\begin{corollary} \label{complete}
The vertex-edge graph of $P(G)$ is complete if and only if for all $I \subseteq [t]$: 
$d_I = d$ or $d_{I^c} = d$.
\end{corollary}
\begin{proof}
The vertex-edge graph is not complete if and only if there is a decomposable element in $G$. 

If $g^k$ is decomposable, then the subset $I \subseteq [t]$ from Proposition \ref{2crit} has obviously the property $d_I < d$ and $d_{I^c} < d$. 
On the other hand, if there is some $I \subseteq [t]$ with $d_I < d$ and $d_{I^c} < d$, then 
$d_I$ and $d_{I^c}$ do not divide $d_I + d_{I^c}$, since $\lcm(d_I,d_{I^c}) = d$. Hence 
$g^{d_I + d_{I^c}}$ is decomposable due to Proposition \ref{2crit}. 
\end{proof}
This criterion may be used to give many interesting high-dimensional examples of such polytopes, cf. Section 3. 
We finish the section with the following conjecture, which holds for $l=1$ by the previous corollary and 
has been experimentally checked in many cases. 
\begin{conjecture}\label{neighb}
Let $l \geq 1$. 
If $d_I=d$ for all $ I\subseteq [t]$ with $|I|\ge \lceil\frac{t}{l+1}\rceil$, then $P(G)$ is {\em $(l+1)$-neighborly}, 
i.e, every subset of at most $l+1$ vertices of $P(G)$ forms the vertex set of a face.
\end{conjecture}

\section{Cyclic permutation groups with few orbits}

Cyclic permutation groups with one orbit are completely described in Corollary~\ref{simplexkrit}. In this section
we study those with two or more orbits.

\subsection{Projection map and joins}
\label{proj-section}

Let $G \leq S_n$ be a permutation group with orbits $O_1, \ldots, O_t$. 
Let $g \in G$. The permutation matrix $\prmat(g)$ 
has a blockdiagonal-structure corresponding to the $t$ orbits:\

\label{proj-subs}

\bigskip

\begin{figure}[h]
  \centering
  \begin{tikzpicture}[scale=.4]

    \foreach \x/\y/\z in {x00/0/0,x30/3/0,x60/6/0,x90/9/0,x09/0/9,x39/3/9,x69/6/9,x99/9/9,x33/3/3,x06/0/6,x93/9/3,x66/6/6,x62/6/2,x35/3/5,x08/0/8,x92/9/2,x65/6/5,x38/3/8,x36/3/6,x63/6/3,xm11/-1/1,xm18/-1/8,x101/10/1,x108/10/8} { %
      \coordinate (\x) at (\y,\z);%
    } %
    \coordinate (v1) at (barycentric cs:x08=.25,x38=.25,x09=.25,x39=.25);%
    \coordinate (o1) at (barycentric cs:x06=.25,x36=.25,x09=.25,x39=.25);%
    \coordinate (v2) at (barycentric cs:x35=.25,x65=.25,x36=.25,x66=.25);%
    \coordinate (o2) at (barycentric cs:x33=.25,x63=.25,x36=.25,x66=.25);%
    \coordinate (v3) at (barycentric cs:x63=.25,x62=.25,x92=.25,x93=.25);%
    \coordinate (o3) at (barycentric cs:x60=.25,x63=.25,x90=.25,x93=.25);%

    \foreach \x/\y/\z in {y00/16/4,y01/16/5,y30/19/4,y31/19/5,y60/22/4,y61/22/5,y90/25/4,y91/25/5} {%
      \coordinate (\x) at (\y,\z);%
    }%

    \coordinate (w1) at (barycentric cs:y00=.25,y01=.25,y31=.25,y30=.25);%
    \coordinate (w2) at (barycentric cs:y30=.25,y31=.25,y61=.25,y60=.25);%
    \coordinate (w3) at (barycentric cs:y60=.25,y61=.25,y91=.25,y90=.25);%

    \coordinate (p1) at (11,4.5);%
    \coordinate (p2) at (15,4.5);%

    \fill [fill=red!10] (x06) -- (x09) -- (x39) -- (x36);%
    \fill [fill=red!20] (x08) -- (x09) -- (x39) -- (x38);%
    \fill [fill=red!10] (x33) -- (x36) -- (x66) -- (x63);%
    \fill [fill=red!20] (x35) -- (x36) -- (x66) -- (x65);%
    \fill [fill=red!10] (x60) -- (x63) -- (x93) -- (x90);%
    \fill [fill=red!20] (x62) -- (x63) -- (x93) -- (x92);%

    \fill [fill=red!20] (y00) -- (y01) -- (y91) -- (y90);%

    \draw (x108) arc (0:30:2);%
    \draw (x101) -- (x108);%
    \draw (x101) arc (0:-30:2);%
    \draw (xm18) arc (180:150:2);%
    \draw (xm11) -- (xm18);%
    \draw (xm11) arc (180:210:2);

    \draw (x06) -- (x09) -- (x39);%
    \draw (x60) -- (x90) -- (x93);%
    \draw (x33) -- (x39);%
    \draw (x60) -- (x66);%
    \draw (x33) -- (x93);%
    \draw (x06) -- (x66);%
    \draw (x62) -- (x92);%
    \draw (x35) -- (x65);%
    \draw (x08) -- (x38);%

    \draw (y00) -- (y01) -- (y91) -- (y90) -- (y00);
    \draw (y30) -- (y31);
    \draw (y60) -- (y61);

    \draw [thick,->] (p1) -- (p2);

    \node at (v1) {$v_1$};%
    \node at (o1) {$O_1$};%
    \node at (v2) {$v_2$};%
    \node at (o2) {$O_2$};%
    \node at (v3) {$v_3$};%
    \node at (o3) {$O_3$};%
    \node at (w1) {$v_1$};%
    \node at (w2) {$v_2$};%
    \node at (w3) {$v_3$};%
    \node [above=5pt] at (barycentric cs:p1=1,p2=1) {$\pi$};
  \end{tikzpicture}
  
  \caption{A permutation matrix with three orbits and the relevant first rows of each block}
\label{fig:permmatrix}
\end{figure}
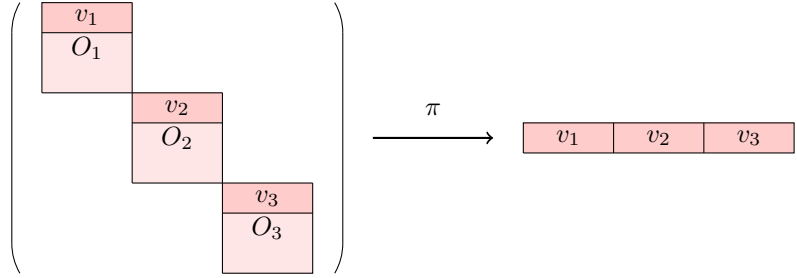
\bigskip

For any such matrix let $v_i(M) \in \R^{|O_i|}$ be the first row in the $i$th block. 
Since any element in $\aff(P(G))$ has such a block-diagonal-structure, we 
define the linear projection map 
\[\pi \;:\; \aff(P(G)) \to \{x \in \R^n \;:\; \sum_{i=1}^n x_i = t\}\]
by projecting any matrix $M$ onto $(v_1(M), \ldots, v_t(M))$.\

Let us assume that $G$ {\em acts cyclic on every orbit}, i.e., for
each $i \in [t]$ the quotient group $G/K_i$ is cyclic, where
$K_i$ is the kernel of the action of $G$ on $O_i$ (the set of group
elements which leave each element in $O_i$ fixed).
Under this assumption, $\pi$ is  a lattice isomorphism of $P(G)$ onto
its image in $\R^n$.

In some cases one can say more. For this let us give the following definition.

\begin{definition}
Let us assume that the polytope $P$ lies in an affine hyperplane of $\R^n$. Then $P$ is 
a {\em join} of polytopes $P_1, \ldots, P_s$, if 
$P$ is the convex hull of $P_1, \ldots, P_s$, and $\lin(P)= \oplus_{i=1}^s \lin(P_i)$. 
We say, $P$ a {\em $\Z$-join}, if $\lin(P) \cap \Z^n= \oplus_{i=1}^s \lin(P_i) \cap \Z^n$.
\end{definition}
A typical example is a tetrahedron: it is the join of two disjoint edges.
\begin{lemma}\label{join}
Let $G \leq S_n$ be a permutation group with orbits $O_1, \ldots, O_t$. 
For each $i \in [t]$ let $G_i$ be the stabilizer of an element $k_i \in O_i$.

If $G$ acts cyclic on every orbit, then the permutation polytope $P(G)$ is 
the $[G : H]$-fold $\Z$-join of permutation polytopes $P(H)$, for $H
:= G_1 \cdots G_t \leq G$.
\end{lemma}
\begin{proof}
Let $K_i$ be the kernel of the action of $G$ on $O_i$, $i \in
[t]$. Then, as $G/K_i$ is cyclic, $[G,G] \leq K_i$
for $i \in [t]$. Thus $[G,G] \leq \cap_{i=1}^t K_i = \{e\}$. So $G$ is
abelian. This implies that $K_i = G_i$ for $i \in [t]$  and that $H :=
G_1 \cdots G_t$ is a subgroup of $G$.

Now let $s := [G:H]$, and let $H g_1, \ldots, H g_s$ be  the right
cosets of $G/H$.
For $j \in [s]$ we define $P_j := \pi(P(H g_j)) \cong P(H g_j) \cong
P(H)$, where these are lattice isomorphisms.
It remains to show that $\pi(P(G))$ is the $\Z$-join of $P_1, \ldots,
P_s$. It is clear that $\pi(P(G))$ is the convex hull of $P_1, \ldots,
P_s$.

Let $i,j \in [t]$. We set $k_i^{H g_j} := \{k_i^{h g_j} \,:\, h \in H\}$.
Then it is straightforward to prove that the orbit $O_i$ is partitioned into 
the sets $k_i^{H g_1}, \ldots, k_i^{H g_s}$.
This implies that for $j_1, j_2 \in [s]$ with $j_1 \not= j_2$, 
the vertices of $P_{j_1}$ and $P_{j_2}$ have disjoint
support. Therefore, $\lin(\pi(P(G))) \cap \Z^n= \oplus_{i=1}^s
\lin(P_i) \cap \Z^n$.
\end{proof}

Let us apply this lemma to the cyclic case. Let $g \in S_n$ have cycle
decomposition into cycles of lengths $\ell_1, \ldots, \ell_t$. Then
$G_i$ is generated by $g^{\ell_i}$ for $i \in [t]$.
Let $q := \gcd(\ell_1, \ldots, \ell_t)$. Hence, $H$ is generated by $g^q$. Therefore, 
$[G:H] = q$. Moreover, since $g^q$ has a cycle decomposition into cycles of the lengths $\ell_1/q, \ldots, \ell_t/q$ (with possible repetitions), 
we see that $H$ is effectively equivalent to a permutation group $H'$ generated by a product of $t$ disjoint 
cycles of lengths $\ell_1/q, \ldots, \ell_t/q$. By Theorem~\ref{thm:stablyEquivalence} 
we have $P(H) \cong P(H')$, and this projection map is even a lattice isomorphism. Lemma~\ref{join} implies that it suffices to consider the case 
$\gcd(\ell_1, \ldots, \ell_t) = 1$ in order to understand the complete face structure of $P(G)$. Together with  Proposition~\ref{disjoint} we obtain the following result.

\begin{prop}\label{two}
Let $G = \langle g\rangle  \leq S_n$ where $g$ has a cycle decomposition into two cycles of lengths $\ell_1, \ell_2$. 
We set $q := \gcd(\ell_1,\ell_2)$. Then $P(G)$ is the $q$-fold $\Z$-join of 
\[\Delta_{\frac{\ell_1}{q}-1} \times \Delta_{\frac{\ell_2}{q}-1},\]
where $\Delta_l$ is the $l$-dimensional unimodular simplex.
\end{prop}
The dimension of this polytope is $\ell_1 + \ell_2 -
\gcd(\ell_1,\ell_2) - 1$ in accordance with the dimension formula
given in Corollary~\ref{formula}. It has $\lcm(\ell_1,\ell_2)$
vertices and $\ell_1+\ell_2$ facets.

Ehrhart polynomials count lattice points in multiples of a lattice polytope \cite{Stanley,Ehrhart62}. 
In \cite{BDO11} Ehrhart polynomials of certain permutation polytopes are computed, including the case of a cyclic permutation group 
with one orbit. In Corollary~\ref{ehrhart} below, we will provide an explicit formula for the generating function of the Ehrhart polynomial of a permutation polytope 
associated to a cyclic permutation group with two orbits. For this purpose, we need a folklore result 
for which we couldn't find a suitable reference.

\begin{lemma}\label{h-star}
Let $\Delta_a$, $\Delta_b$ be two unimodular simplices in lattices $N_1$, $N_2$ respectively. Then 
\[\sum_{k=0}^\infty |(k (\Delta_a \times \Delta_b)) \cap (N_1 \oplus N_2)| \; t^k=\frac{\sum_{i=0}^{\min(a,b)} \binom{a}{i} \binom{b}{i} t^i}{(1-t)^{a+b+1}}.\]
\end{lemma}
\begin{proof}
By definition, we have 
\[|(k (\Delta_a \times \Delta_b)) \cap (N_1 \oplus N_2)| = \binom{k+a}{a}\binom{k+b}{b}.\]
Since $t^i/(1-t)^{a+b+1} = \sum_{k=0}^\infty \binom{k+a+b-i}{a+b} t^k$ (e.g., \cite{Stanley}), it remains to show that
\[\binom{k+a}{a}\binom{k+b}{b} = \sum_{i=0}^{\min(a,b)} \binom{a}{i} \binom{b}{i} \binom{k+a+b-i}{a+b}.\]
This is a well-known binomial identity. For instance, it can be deduced from (5.28) in \cite{Knuth}.
\end{proof}

It is also possible to prove the previous result by computing the
$h$-vector from a shelling of the staircase triangulation of the product of 
two simplices, cf. \cite{Roots}. 

\begin{corollary}\label{ehrhart}
Let $G = \langle g\rangle  \leq S_n$ where $g$ has a cycle decomposition into two cycles of lengths $\ell_1, \ell_2$. 
We set $q := \gcd(\ell_1,\ell_2)$. Then
\[\sum_{k=0}^\infty |(k P(G)) \cap \Z^{n^2}| \; t^k=
\frac{\left(\sum_{i=0}^{\min(\frac{\ell_1}{q}-1,\frac{\ell_2}{q}-1)} \binom{\frac{\ell_1}{q}-1}{i} \binom{\frac{\ell_2}{q}-1}{i} t^i\right)^q}{(1-t)^{\ell_1 + \ell_2 - \gcd(\ell_1,\ell_2)}}.\]
\end{corollary}
\begin{proof}
By Lemma 1.3 in \cite{HT09} the enumerator polynomials of the Ehrhart generating series of $\Z$-joins are multiplicative. Hence, the result follows from Lemma~\ref{h-star} and Proposition~\ref{two}.
\end{proof}

\subsection{Permutation polytopes of cyclic groups with three orbits}

Let $G = \group g \le S_n$ be a cyclic permutation group of order
$d$. In Corollary~\ref{simplexkrit} and Proposition~\ref{two}
we completely described the combinatorial type of $P(G)$ when $G$ has
at most two orbits.
In the case of three orbits, we cannot present a corresponding
result. Here the situation is much more complicated.
In the following we will focus on one crucial case. For three pairwise
coprime numbers $a,b,c\in \N_{\geq 2}$ let $z_{ab}$, $z_{ac}$ and
$z_{bc}$ be three disjoint cycles of lengths $ab$, $ac$, and $bc$,
respectively. We define
\[P(a,b,c):=P(\group{z_{ab}z_{ac}z_{bc}}).\]
By Corollary~\ref{formula}, $P(a,b,c)$ has dimension $ab + ac + bc - a
- b - c$.  The number of vertices is $a b c$.  By
Corollary~\ref{complete} all of these polytopes have a complete 
vertex-edge graph.  In Table~\ref{tab:facet-numbers} we present the
number of facets which we were able to compute using
\texttt{polymake}~\cite{polymake99}. Note that one very quickly
reaches the limits of computational power.
\begin{table}[ht]
  \centering
  \begin{tabular}{r|rrrrr} 
    $(a,b,c)$ & $(2,3,5)$ & $(2,3,7)$ & $(2,5,7)$ & $(2,5,9)$ & $(3,4,5)$\\ 
    \midrule
    $\#$ dimension& 21 & 29 & 45& 57& 35\\ 
    \midrule
    $\#$ vertices& 30 & 42 & 70 & 90 & 60\\ 
    \midrule
    $\#$ facets& 211& 797& 3839& 15373& 29387
  \end{tabular}
  \caption{Dimension, vertices and facets of $P(a,b,c)$}
  \label{tab:facet-numbers}
\end{table}

The following result shows that the number of facets grows indeed exponentially.
\begin{theorem}\label{main-theo}
  Let $a,b,c \geq 2$ be pairwise coprime integers.

  Then $P(a,b,c)$ has at least $\frac12(2^a-2)(2^b-2)(2^c-2)+ab+ac+bc$
  facets.
\end{theorem}
For $a=2$ this result seems to be optimal, see
Table~\ref{tab:facet-numbers}. This motivates the following
conjecture. Note that the bound in the theorem is not sharp for $a=3$.
\begin{conjecture}
  Let $b,c \geq 3$ be odd and coprime. Then the number of facets of
  $P(2,b,c)$ equals $(2^b-2) (2^c-2) + 2b +2c +bc$.  
\end{conjecture}
The proof of Theorem~\ref{main-theo} will be given in the remainder of
this paper.  We are going to describe explicitly a set of facets for
$P(a,b,c)$.

\subsubsection{Setting and outline of the proof of Theorem~\ref{main-theo}}

From  now  on let $a,b,c$ be pairwise coprime positive integers. Let $n = ab+ac+bc$, and $G \le S_n$ be generated by the product $g$ of
three disjoint cycles of lengths $ab$, $ac$ and $bc$. In the following
we will always identify $P(a,b,c)$ with $\pi(P(a,b,c))$, as described
in ~\ref{proj-subs}. In particular, any element of $G$ will be
considered as a vector in $\R^{ab+ac+bc}$ having coordinates $x_0,
\ldots, x_{ab-1}$, $y_0, \ldots, y_{ac-1}$, and $z_0, \ldots, z_{bc-1}$. For
$u\in\R^{ab+ac+bc}$, we let $\pi_x(u)$, $\pi_y(u)$, and $\pi_z(u)$ be
the projections onto the $x$-, $y$-, and $z$-coordinates,
respectively.

\begin{prop}
The inequalities
\begin{align*}
  x_i&\ge 0& y_j&\ge 0& z_k&\ge 0
\end{align*}
define facets of $P(a,b,c)$.
\end{prop}
\begin{proof}
It suffices to prove that these faces are facets. For this, we will show
that for any vertex $g^m$ outside of such a face $F$ we can write 
\[\hat G := \frac{1}{abc} \sum_{g \in G} g\] 
as an affine combination of vertices of the
face together with the given vertex.

Up to symmetry, we may assume that the face $F$ of concern is given by
$x_1 \ge 0$. In particular, it contains all vertices $g^k$ such that $k$ is divisible by $a$ or $b$. 
The vertices outside of $F$ are of the form $g^m$ for $m\equiv 1 \ (\mod\ ab)$. 
Again, up to symmetry, we can choose $m$ such that $m\equiv 0 \ (\mod\ c)$. 
Now, Table~\ref{tab:vb} gives the coefficients of $\hat G$ as an
affine combination of all vertices $g^k$ such that $k$ is divisible by $a$, $b$ or $c$. 

Here is how the reader can check its validity: 
For instance, the projection on the $x$-coordinates of $abc \ \hat G$ 
equals $(c \cdots c) \in \R^{ab}$. Let's consider the $x$-coordinate corresponding to $0 \ (\mod\ a)$ and 
$1 \ (\mod\ b)$. There are $c$ vertices $g^k$ in this equivalence class, $c-1$ not divisible by $c$ and one divisible by $c$. 
By the first and fifth rows of Table~\ref{tab:vb}, this coordinate of the affine combination equals 
\[(c-1) a+(a+c-ac)=c.\]
In the same manner, the statement can be verified for any coordinate.
\newcommand{\yes}{\text{yes}}
\newcommand{\no}{\text{no}}
\begin{table}[t]
  \centering
  \begin{tabular}{>{$}c<{$}>{$}c<{$}>{$}c<{$}|>{$}c<{$}|>{$}c<{$}>{$}c<{$}>{$}c<{$}}
    \multicolumn{3}{c|}{k \text{ divisible by}} &
    \multicolumn{3}{l}{\quad\raisebox{-3mm}[0pt][0pt]{coefficient of $g^k$ times $abc$}} 
    & \quad\raisebox{-3mm}[0pt][0pt]{no.\ vertices of this type}\\
    a & b & c & \\
    \hline
    \yes & \no & \no & a & (b-1)(c-1)\\ 
    \no & \yes & \no & b & (a-1)(c-1)\\ 
    \no & \no & \yes & c & (a-1)(b-1)\\ 
    \yes & \yes & \no & a+b-ab & c-1\\ 
    \yes & \no & \yes & a+c-ac & b-1\\ 
    \no & \yes & \yes & b+c-bc & a-1\\ 
    \yes & \yes & \yes & abc-ab-ac-bc+a+b+c & 1\\ 
  \end{tabular}\\\mbox{ }
  \caption{Coefficients of the vertex barycenter}
\label{tab:vb}
\end{table}
\end{proof}

We say that a facet is \emph{essential}, if it is not of the type $x_i
\ge 0$, $y_j \ge 0$, or $z_k \ge 0$.
There are $n=ab+ac+bc$ non-essential facets.  We want to define a large family of essential facets of $P(a,b,c)$. 
The next subsection defines a certain class of subsets of $[[abc]]$ via projections onto 
the $x$-, $y$-, and $z$-coordinates. In Lemma~\ref{face-lemma} we give
a general criterion when such a set defines a face of $P(a,b,c)$. The final subsection gives an explicit construction of sets that
satisfy the conditions of the lemma. We prove that our vertex sets
define facets and count their number.

\subsubsection{Faces as unions of preimages of projection maps}

Throughout, we will identify
$[[abc]]$ and $G$ via the natural bijection $i \mapsto g^i$.
The Chinese remainder theorem yields a bijection between
$[[abc]]$ and $[[a]] \times [[b]] \times [[c]]$ by mapping $k$ to $(k
\ (\mod\ a), k \ (\mod\ b), k \ (\mod\ c))$. In the same way, we
identify $[[ab]]$ and $[[a]] \times [[b]]$,  $[[ac]]$ and $[[a]]
\times [[c]]$, and $[[bc]]$ and $[[b]] \times [[c]]$.

To any proper subset $S_x \subsetneq [[ab]]$ we associate a subset of
$[[abc]]$ via
\[F_x(S_x) := \pi_x^{-1}(\{e_i \;:\; i \in S_x\}) \ = \ \bigcup\limits_{x
  \in S_x} x \times [[c]] \ \subsetneq \ [[abc]]\;,\] where $e_0,
\ldots, e_{ab-1}$ is the standard basis of $\R^{ab}$. This is (the
vertex set of) a face of $P(a,b,c)$, given by setting $x_i=0$ for $i
\not\in S_x$.  Similarly, we define $F_y(S_y)$ and $F_z(S_z)$ for
subsets $S_y\subsetneq [[ac]]$ and $S_z\subsetneq [[bc]]$. 

In the following we want to consider unions of the form $F_x(S_x) \cup
F_y(S_y) \cup F_z(S_z)$ for $S_x\subsetneq [[ab]]$, $S_y\subsetneq [[ac]]$, $S_z\subsetneq [[bc]]$. In general, this is not the vertex set of a
face. However, the following lemma gives a sufficient criterion.

\begin{lemma}\label{face-lemma}
  Let $S_x \subsetneq [[ab]]$, $S_y\subsetneq [[ac]]$ and $S_z
  \subsetneq [[bc]]$.  If 
  \begin{align}
    F_x(S_x) \cap F_y(S_y) \cap F_z(S_z) = \emptyset\,,\label{eq:face-lemma:1}
  \end{align}
and if for all permutations $(i,j,k)$ of $(x,y,z)$ 
\begin{align}
  F_i(S_i) \ \cap \pi_k^{-1}(\pi_k(F_i(S_i) \cap F_j(S_j))) \ \subseteq
  \ F_j(S_j)\,,\label{eq:face-lemma:2}
\end{align}
  then $F_x(S_x) \cup F_y(S_y) \cup F_z(S_z)$ is the vertex set of a (not
  necessarily proper) face of $P(a,b,c)$.
\end{lemma}
\begin{proof}
The first assumption implies that 
\begin{align*}
  S_x &\cap \pi_x(F_y(S_y) \cap F_z(S_z)) = \emptyset\,,\\
  S_y &\cap \pi_y(F_x(S_x) \cap F_z(S_z)) = \emptyset\,,\\
  \text{and}\qquad S_z &\cap \pi_z(F_x(S_x) \cap F_y(S_y)) = \emptyset\,.
\end{align*}
We define a functional $\lambda = (\vlx, \vly, \vlz) \in \R^n$ in the
following way. Let $I_x:=[[ab]], I_y:=[[ac]]$ and $I_z:=[[bc]]$. For
all permutations $(i,j,k)$ of $(x,y,z)$ we define 
\begin{align*}
  \lambda^{(i)}_m&:=
  \begin{cases}
    -1\qquad  & m \in S_i\\
    \phantom{-}1\qquad & m \in \pi_i(F_j(S_j) \cap F_k(S_k))\\
    \phantom{-}0\qquad & \text{else}\;.
  \end{cases}
\end{align*}
Let $\pro{\cdot}{\cdot}$ by the standard scalar product on $\R^n$ and
$v \in G$.  Using assumptions (\ref{eq:face-lemma:1}) and
(\ref{eq:face-lemma:2}) it is straightforward to check that
$\pro{\lambda}{v} \geq -1$, with equality if and only if $v \in F_x(S_x)
\cup F_y(S_y) \cup F_z(S_z)$.
\end{proof}

\subsubsection{Explicit constructions of facets}

\begin{prop}\label{checker}
Given three non-trivial subsets 
$\emptyset \neq I \subsetneq [[a]]$,
$\emptyset \neq J \subsetneq [[b]]$, and
$\emptyset \neq K \subsetneq [[c]]$, 
the set 
\[([[a]] \times [[b]] \times [[c]])
\setminus (I \times J \times K) 
\setminus (I^c \times J^c \times K^c)\]
is the set of vertices of a facet of $P(a,b,c)$.
\end{prop}
\begin{proof}
We set
\begin{alignat*}{4}
S_x &:= I \times J^c &&\cup I^c \times J &&\subset [[a]] \times [[b]]
  \cong [[ab]]\,,\\
S_y &:= I \times K^c &&\cup I^c \times K &&\subset [[a]] \times [[c]]
  \cong [[ac]]\,\\
S_z &:= J \times K^c &&\cup J^c \times K &&\subset [[b]] \times [[c]]
  \cong [[bc]]\,\\
\end{alignat*}
Then $S_x$, $S_y$, $S_z$ satisfy the conditions of
Lemma~\ref{face-lemma}. The resulting face $F_x(S_x) \cup F_y(S_y) \cup F_z(S_z)$ has 
the vertex set $V$ as given in the statement. We claim that this face is, in fact, a facet. To prove this claim, 
let $v_0 \not\in V$ be an additional vertex of $P(a,b,c)$. We show
that any other vertex $v_1$ of $P(a,b,c)$ can be written as an affine
combination of elements of $V$ together with $v_0$.

As before, we identify the elements of $G$ with triples $(i,j,k) \in
[[a]]\times[[b]]\times[[c]]$. We can assume that $v_0 = (i_0,j_0,k_0) \in I \times J \times K$.
Then either $v_1 = (i_1,j_1,k_1) \in I \times J \times K$ as well, or
$v_1 \in I^c \times J^c \times K^c$.

In the latter case, we see that we have 
$v_1 = v_0 
- (i_0,j_0,k_1)
- (i_0,j_1,k_0)
- (i_1,j_0,k_0)
+ (i_1,j_1,k_0)
+ (i_1,j_0,k_1)
+ (i_0,j_1,k_1)$,
where the last six vertices all belong to $V$. When verifying this statement, the reader should beware that 
this is actually a sum of elements in $\R^{ab+ac+bc}$.

In the former case, we choose $v_2 \in I^c \times J^c \times K^c$, and
construct combinations $v_0 = v_2 + w_0$, $v_1 = v_2 + w_1$, where
$w_0$ and $w_1$ are combinations of elements of $V$ with vanishing
coefficient sum. But then $v_1 = v_0 - w_0 + w_1$ yields the desired
affine representation.
\end{proof}

Finally, let us count the number of different facets we obtain in this
way. We have $(2^a-2)(2^b-2)(2^c-2)$ different choices for $I,J,K$.
Simultaneously exchanging all three sets by their complements yields
the same facet, so the facet depends only on the pairs $(I,I^c)$,
$(J,J^c)$ and $(K,K^c)$.  On the other hand, the set $S:=(I\times
J\times K)\cup (I^c\times J^c\times K^c)$ already determines these
pairs: If $(i,j,k)\in S$, then either $I=\{i'\in [[a]]\mid (i',j,k)\in
S\}$ or $I^c=\{i'\in [[a]]\mid (i',j,k)\in S\}$, and similarly for
$(J,J^c)$ and $(K,K^c)$. Hence, we get $(2^a-2)(2^b-2)(2^c-2)/2$
different facets of this type, and all of these facets are essential
by construction. This finishes the proof of Theorem~\ref{main-theo}.

\begin{remark}
\label{comment}
It is possible to trace the inequalities described in
Lemma~\ref{face-lemma} back to the `cycle inequalities' in
\cite{SJ08}: face inequalities of so-called {\em marginal polytopes}.
In particular, the `checkerboard inequalities' in
Proposition~\ref{checker} may be found in that paper.
However, it is not shown in \cite{SJ08} that they actually define
facets. The precise relation of permutation polytopes to marginal
polytopes \cite{KWA09,SJ08} will be investigated in an upcoming paper
\cite{Marginalspaper}.
\end{remark}

\providecommand{\bysame}{\leavevmode\hbox to3em{\hrulefill}\thinspace}
\providecommand{\MR}{\relax\ifhmode\unskip\space\fi MR }
\providecommand{\MRhref}[2]{%
  \href{http://www.ams.org/mathscinet-getitem?mr=#1}{#2}
}
\providecommand{\href}[2]{#2}

\end{document}